\documentclass[leqno,12pt]{article}
\usepackage{amsfonts,amsmath,amssymb}
\usepackage{amsthm}
\usepackage{enumerate}
\usepackage{graphics}
\usepackage{float}%
\usepackage{enumerate}%
\usepackage[hypertex]{hyperref}%
\newcommand{\scirc}	{\raise2pt\hbox{${}_\circ$}}

\newcommand{\A}
           {\widetilde{\mathbb{A}}}

\newcommand{\KA}[2]
           {\widetilde{K}_{{#1},{#2}}^{\mathbb{A}}}
\newcommand{\ka}[2]
           {{K}_{{#1},{#2}}^{\mathbb{A}}}

\newcommand{\KC}[2]
           {\widetilde{K}_{{#1},{#2}}^{\mathbb{C}}}
\newcommand{\cA}[2]
           {{\mathbb{A}}_{{#1},{#2}}}

\newcommand{\T}[2]
           {\widetilde{\mathbb{T}}_{{#1},{#2}}}
\newcommand{\KT}[2]
           {\widetilde{K}_{{#1},{#2}}^{\mathbb{T}}}
\newcommand{\tcA}[2]
           {\widetilde{\mathbb{A}}_{#1,#2}}
\newcommand{\tcC}[2]
           {\widetilde{\mathbb{C}}_{#1,#2}}
\newcommand{\tcT}[2]
           {\widetilde{\mathbb{T}}_{#1,#2}}
\newtheorem{thmalph}{Theorem}

\newtheorem{theorem}{Theorem}[section]
\newtheorem{proposition}[theorem]{Proposition}

\theoremstyle{remark}
\newtheorem{remark}[theorem]{Remark}
\theoremstyle{definition}
\newtheorem{example}[theorem]{Example}
\newtheorem{definition}[theorem]{Definition}
\newtheorem{problem}[theorem]{Problem}
\newtheorem{projectalph}[thmalph]{Project}

\numberwithin{equation}{section}
\def\Sol{{\mathcal {S}}\!{\text{\it{ol}}}}
\begin{document}
\title{
F-method for symmetry breaking operators
\\
\textit{\large 
Dedicated to  Professor Michael Eastwood for his 60th birthday}}
\author{%
 Toshiyuki KOBAYASHI\\[\smallskipamount]
Kavli IPMU (WPI), and Graduate School of Mathematical Sciences \\
the University of Tokyo\\
\normalsize
\textit{E-mail address}: \texttt{toshi@ms.u-tokyo.ac.jp}
}
\date{} 

\maketitle

\begin{abstract}
We provide some insights in the study of branching problems of reductive groups, and a method of investigations into symmetry breaking operators.
First, we give geometric criteria
 for finiteness property of linearly independent
 continuous (respectively, differential)
 operators that intertwine two induced representations
 of reductive Lie groups and their reductive subgroups.  
Second,
 we  extend the \lq{F-method}\rq\
 known for local operators to non-local operators.  
We then illustrate the idea by concrete examples in conformal geometry, 
 and explain how the F-method works for detailed
analysis of symmetry breaking operators,
 e.g.,
 finding functional equations
and explicit residue formulae of \lq regular\rq\ symmetry breaking
 operators  with meromorphic parameters.
\end{abstract}

\medskip
\noindent
\textit{2010 MSC:}
Primary
          22E46; 
Secondary 
33C45, 53C35

\medskip
\noindent
\textit{%
Key words and phrases\/}: 
branching law, reductive Lie group,
symmetry breaking,
parabolic geometry,
conformal geometry,
Verma module, F-method.  

\setcounter{tocdepth}{1}
\tableofcontents

\section{Introduction}

Let $G$ be a real reductive linear Lie group, 
 and $P$ a parabolic subgroup.  
Associated to a finite-dimensional representation 
$\lambda:P \to GL_{\mathbb{C}}(V)$, 
 we define a homogeneous vector bundle
 ${\mathcal{V}}:=G \times_P V$
 over the (generalized) real flag variety
 $X:=G/P$.  
Then the group $G$ acts continuously on the space
 $C^{\infty}(X, {\mathcal{V}})$
 of smooth sections,
 which is endowed 
 with the natural Fr{\'e}chet topology
 of uniform convergence
 of finite derivatives.  

Suppose $G'$ is an (algebraic) reductive subgroup
 of $G$, 
 $P'$ is a parabolic subgroup of $G'$, 
 and $Y:=G'/P'$.  
For a finite-dimensional representation 
 $\nu: P' \to GL_{\mathbb{C}}(W)$,
 we define similarly a homogeneous vector bundle
 ${\mathcal{W}}:=G' \times_{P'}W$
 over $Y$,
 and form a continuous representation
 of $G'$ on $C^{\infty}(Y,{\mathcal{W}})$.  
We denote by $\operatorname{Hom}_{G'}(C^{\infty}(G/P, {\mathcal{V}}),
C^{\infty}(G'/P', {\mathcal{W}}))$
 the space of continuous $G'$-homomorphisms
(\textit{symmetry breaking operators}).

Assume that 
\begin{equation}
\label{eqn:PGP}
P' \subset P \cap G'.   
\end{equation}
Then we have a natural $G'$-equivariant morphism 
 $\iota: Y \to X$.  
With this morphism $\iota$, 
 we can define a continuous linear operator
 $T:C^{\infty}(X,{\mathcal{V}}) \to C^{\infty}(Y,{\mathcal{W}})$
 to be a {\it{differential operator}}
 in a wider sense 
 than the usual 
 by the following local property:
\[
\iota (\operatorname{Supp}(Tf)) \subset \operatorname{Supp}f
\quad
\text{ for any }
f \in C^{\infty}(X,{\mathcal{V}}).  
\]
In the case where $Y=X$ with $\iota$ the identity map, 
 this definition is equivalent to
 that $T$ is a differential operator 
in the classical sense
 by Peetre's theorem.  
We shall write 
 $\operatorname{Diff}_{G'}(C^{\infty}(G/P, {\mathcal{V}}),
C^{\infty}(G'/P', {\mathcal{W}}))$
 for the space of $G'$-intertwining differential operators.  

The object of our study is local and non-local symmetry breaking operators:
\begin{equation}
\label{eqn:DC}
\operatorname{Diff}_{G'}(C^{\infty}(X,{\mathcal{V}}),
 C^{\infty}(Y,{\mathcal{W}}))
\subset 
\operatorname{Hom}_{G'}(C^{\infty}(X,{\mathcal{V}}),
 C^{\infty}(Y,{\mathcal{W}})).  
\end{equation}

We consider the following:
\begin{projectalph}
Construct {\it{explicitly}} differential 
(local)/ continuous (non-local) symmetry breaking operators, 
 and classify them.  
\end{projectalph}

In the setting where $X=Y$ and $G=G'$, 
 the Knapp--Stein intertwining operators 
 are a basic example of non-local intertwining operators.  
On the other hand,
 the existence condition
 for non-zero homomorphisms 
between generalized Verma modules
 has been studied algebraically 
 by many authors
 in various cases
 (see \cite{Ma} and references therein), 
 which is in turn equivalent to the existence condition
 for non-zero local $G$-intertwining operators
 by the duality between Verma modules
 and principal series representations (e.g. \cite{xhaja}). 
However, even in this usual setting,
 it is already a non-trivial matter
 to write down explicitly the operators,
 and a complete classification of such operators
 is far from being solved for reductive groups $G$
 in general.

In the setting that we have in mind,
 namely,
 where $G' \subsetneqq G$, 
 we face with branching problems
 of irreducible representations
 of the group $G$
 when we restrict them 
 to the subgroup $G'$
 (or branching problems for Lie algebras).  
The study of the restriction
 of infinite-dimensional representations
 is difficult in general even as \lq{abstract analysis}\rq, 
involving wild features
 (\cite{xkiai, K12}).  
On the other hand,
 Project A asks further 
\lq{concrete analysis}\rq\ of the restriction.  
Nevertheless,
 there have been recently some explicit results
 \cite{CKOP, Juhl, xkhelgason, KOSS, KP, KS}
 for specific situations 
where $X \neq Y$ and $G \neq G'$, 
 in relation to Project A.

\vskip 1pc
In light of the aforementioned state
 of the art,
 we first wish to clarify
  what are reasonably general settings for Project A
  and what are their limitations,
 not coming from the existing 
 technical difficulties,
 but
 from purely representation theoretic constraints.
For this, 
 we remember
 that the spaces of symmetry breaking operators
are not always finite-dimensional
 if $G \neq G'$,
 and consequently,
 the subgroup $G'$ may lose a good control
 of the irreducible $G$-module (\cite{xkiai}).
Thus we think that Project A should be built
 on a solid foundation
 where the spaces of local/non-local symmetry breaking operators
 are at most finite-dimensional, 
and preferably, 
 of uniformly bounded dimensions, 
 or even one-dimensional.
In short, we pose:
\begin{projectalph}
Single out appropriate settings
 in which Project A makes sense.  
\end{projectalph}

In Section \ref{sec:finitethm}, 
 we discuss Project B by applying recent progress
on the theory of restrictions of representations
 \cite{K08,K12, mfbundl, xktoshima},
 and provide concrete geometric conditions
 that assure
 the spaces in \eqref{eqn:DC}
 to be finite-dimensional, 
 of uniformly bounded dimensions,
 or at most one-dimensional (multiplicity-free restrictions).

In Sections \ref{sec:Fmethod} to \ref{sec:4}
 we develop a method of Project A by 
  extending the idea of the \lq F-method\rq\ 
 studied earlier for local operators 
 (\cite{xkhelgason, KOSS, KP})
 to non-local operators.

\section{Finiteness theorems of restriction maps}
\label{sec:finitethm}

This section is devoted to Project B.

We begin with the most general case,
 namely,
 the case
 where $P$ and $P'$
 are minimal parabolic subgroups
 of $G$ and $G'$, 
respectively.  
In this case,  the assumption \eqref{eqn:PGP} is automatically satisfied
 (with $P$ replaced by its conjugate,
 if necessary).  
Under the condition, we give finiteness criteria
 for local and non-local operators:
\begin{theorem}
[local operators]
\label{thm:fmdiff}
Assume $\operatorname{rank}_{{\mathbb{R}}}G
=\operatorname{rank}_{{\mathbb{R}}}G'$.  
Then
\begin{equation}
\label{eqn:fmdiff}
\dim \operatorname{Diff}_{G'}
(C^{\infty}(X, {\mathcal{V}}), 
C^{\infty}(Y, {\mathcal{W}}))<\infty
\end{equation}
for all finite-dimensional representations
 $V$ of $P$ and $W$ of $P'$.  
\end{theorem}

\begin{remark}
We give a proof of Theorem~\ref{thm:fmdiff}
 for more general parabolic subgroups $P$ and $P'$,
 see Theorem~\ref{thm:finVerma} below.  
We note that the real rank assumption of Theorem~\ref{thm:fmdiff}
 is relaxed
in Theorem~\ref{thm:finVerma}.  
\end{remark}

\begin{theorem}
[non-local operators] 
\label{thm:fimHom}
{\rm{1)}}\enspace
{\rm{(finite multiplicity)}}\enspace
The following two conditions
 on the pair $(G,G')$ of real reductive groups
 are equivalent:
\begin{enumerate}
\item[{\rm{(i)}}]
For all finite-dimensional representations
 $V$ of $P$ and $W$ of $P'$, 
\begin{equation}
\label{eqn:fmHom}
\dim \operatorname{Hom}_{G'}(C^{\infty}(G/P, {\mathcal{V}}), 
C^{\infty}(G'/P', {\mathcal{W}}))
 < \infty.  
\end{equation}
\item[{\rm{(ii)}}]
There exists an open $P'$-orbit
 on the (generalized) real flag variety $G/P$.  
\end{enumerate}
\par\noindent
{\rm{2)}}\enspace
{\rm{(uniformly bounded multiplicity)}}\enspace 
Let $G$ be a simple Lie group.  
We write ${\mathfrak {g}}$ and ${\mathfrak {g}}'$
 for the complexifications
 of the Lie algebras
 ${\mathfrak {g}}_{\mathbb{R}}$ and ${\mathfrak {g}}_{\mathbb{R}}'$
 of $G$ and $G'$, 
 respectively.  
Then the following conditions 
 on the pair $(G,G')$ are equivalent:
\begin{enumerate}
\item[{\rm{(i)}}]
\begin{equation}
\label{eqn:ssVW}
\sup_V \sup_W \dim 
\operatorname{Hom}_{G'}(C^{\infty}(X,{\mathcal{V}}),
 C^{\infty}(Y,{\mathcal{W}}))
<\infty,
\end{equation}
where the supremum is taken over all finite-dimensional, 
 irreducible representations
 $V$ of $P$ and $W$ of $P'$,
 respectively.  
\item[{\rm{(ii)}}]
There exists an open $B'$-orbit on the complex flag variety
 $G_\mathbb C/B$, 
 where $B$ and $B'$ are Borel subgroups of complexifications $G_\mathbb C$
 and $G'_\mathbb C$ of $G$ and $G'$
 respectively.
\item[{\rm{(iii)}}]
{\rm{(strong Gelfand pair)}}\enspace
The pair $({\mathfrak {g}},{\mathfrak {g}}')$
 is one of 
 $({\mathfrak {sl}}(n+1,{\mathbb{C}}),
   {\mathfrak {gl}}(n,{\mathbb{C}}))$,
$({\mathfrak {sl}}(n+1,{\mathbb{C}}),
{\mathfrak {sl}}(n,{\mathbb{C}}))$ ($n \neq 1$),
$({\mathfrak {so}}(n+1,{\mathbb{C}}),
{\mathfrak {so}}(n,{\mathbb{C}}))$,
or 
${\mathfrak {g}}={\mathfrak {g}}'$.  
\end{enumerate}
\end{theorem}
Theorem \ref{thm:fimHom} (1) suggests
 that the orbit structure $P'\backslash G/P$ 
 is crucial 
 to understand $\operatorname{Hom}_{G'}
(C^{\infty}(X, {\mathcal{V}}), C^{\infty}(Y, {\mathcal{W}}))$.  
In fact,
 the \lq{regular}\rq\ symmetry breaking operators
 are built on the integral transform attached to open $P'$-orbits
 on $G/P$ \cite{KS},
whereas the closed $P'$-orbit through the origin
 $o=eP/P \in G/P$ is
 the support of the distribution kernel 
 (see Proposition \ref{prop:KT} below)
of differential symmetry
breaking operators \cite{KP}.
The whole orbit structure plays a basic role in
 the classification of all symmetry breaking operators
 in \cite{KS}
 in the setting that we discuss in Section \ref{sec:4}.

A remarkable feature
 of Theorem \ref{thm:fimHom} (2)
 is that the equivalent conditions 
 do not depend on the real form 
 $({\mathfrak{g}}_{\mathbb{R}}, {\mathfrak{g}}_{\mathbb{R}}')$
 but depend only on the complexification
 $({\mathfrak {g}}, {\mathfrak {g}}')$,
 which are obvious from (ii) or (iii)
 but are quite non-trivial from (i).  
On the other hand,
 the equivalent conditions 
 in Theorem \ref{thm:fimHom} (1)
  depend heavily on the real form
 $({\mathfrak {g}}_{\mathbb{R}}, {\mathfrak {g}}_{\mathbb{R}}')$.

Comparing the criteria given in Theorems \ref{thm:fmdiff} and \ref{thm:fimHom},
 we find that the space
 of non-local intertwining operators
 (i.e. the right-hand side of \eqref{eqn:DC})
 are generally much larger
 than that of local operators
 (i.e. the left-hand side of \eqref{eqn:DC}).  
Thus
 \lq{appropriate settings}\rq\
 in Project B will be different
 for local and non-local operators:
\begin{example}
Suppose $(G,G')
=(GL(p+q,{\mathbb{R}}), GL(p,{\mathbb{R}})\times GL(q,{\mathbb{R}}))$.  
Then the finite-dimensionality \eqref{eqn:fmdiff} of local operators
 holds for all $p, q$, 
 whereas the finite-dimensionality \eqref{eqn:fmHom}
 of non-local operators
 fails if $p \ge 2$ or $q \ge 2$.  
Likewise, for the non-symmetric pair
$$
 (G,G')
 =(GL(p+q+r,{\mathbb{R}}), GL(p,{\mathbb{R}})\times GL(q,{\mathbb{R}})
 \times GL(r, \mathbb R)),
$$
the finite-dimensionality \eqref{eqn:fmdiff} still holds for all $p,q, r>0$, 
 whereas \eqref{eqn:fmHom} fails for any $p, q, r>0$.  
\end{example}
\begin{example}
[group case]
\label{ex:2.5}
Let $(G,G')
=(H \times H, \operatorname{diag}H)$
 with $H = O(p,q)$.  
Then the condition (ii) of Theorem \ref{thm:fimHom} (1) holds
 if and only if 
 $\min(p,q) \le 1$.  
Again,
 \eqref{eqn:fmdiff} always holds 
 but \eqref{eqn:fmHom} fails
 if $\min(p,q) \ge 2$.  
Symmetry breaking operators
 for the pair $(H \times H, \operatorname{diag}H)$
 correspond
 to invariant trilinear forms
 of representations of $H$, 
 for which concrete analysis
 was studied, 
 e.g., 
 in \cite{CKOP}
 in the case $H=O(p,1)$.
\end{example}
In a subsequent paper \cite{xKMt},
we shall give a complete classification on the level of the Lie 
algebras of the reductive symmetric pairs 
$(G, G')$  that satisfy  the equivalent conditions
 of Theorem \ref{thm:fimHom} (1). 
(The classification was given earlier
 in \cite{Ksuron}
 in the setting 
 where $(G,G')$ is of the form 
 $(H \times H, \operatorname{diag}H)$, 
 and Example \ref{ex:2.5} is essentially
 the unique example
 satisfying the equivalence conditions
 in this case.)

\vskip 1pc
 Here are some comments
 on the proof of Theorem \ref{thm:fimHom}.  
The implication (ii) $\Rightarrow$ (i)
 in both (1) and (2) of Theorem \ref{thm:fimHom}
 is a direct consequence of 
  the main theorems of \cite[Theorems A and B]{xktoshima}
  which are stated in a more general setting.  
The converse implication
 is proved 
 by using a generalized Poisson transform
 \cite[Theorem 3.1]{xktoshima}.  
The method is based on 
the theory of system of partial differential equations
 with regular singularities and hyperfunction boundary values, 
  developed  by M.~Sato, M.~Kashiwara, T.~Kawai, and T.~Oshima \cite{xkaos}
 and the compactification of the group manifold
 with normal-crossing boundaries.  
Alternatively, 
 the implication (ii) to (i)
 in Theorem \ref{thm:fimHom} (2) could be derived 
  from a recent multiplicity-free theorem 
\cite{xsunzhu}
 and from the classification
 of real forms
 of (complex) strongly Gelfand pairs
 \cite{xkramer}.  
We note 
 that the proof in \cite{xktoshima} does not use any case-by-case argument.

\medskip

We give a proof of Theorem \ref{thm:fmdiff}
 in a more general form below.  
Let us fix some notation.  
A semisimple element $H$ 
 of a complex reductive Lie algebra ${\mathfrak {g}}$
 is said to be {\it{hyperbolic}}
 if the eigenvalues
 of $\operatorname{ad}(H)
 \in {\operatorname{End}}_{\mathbb{C}}({\mathfrak {g}})$
 are all real.  
Given a hyperbolic element $H$, 
 we define subalgebras
 of ${\mathfrak {g}}$ by 
\[
{\mathfrak{n}}^+ \equiv{\mathfrak{n}}^+(H), 
\quad
{\mathfrak{l}} \equiv{\mathfrak{l}}(H),
\quad
 {\mathfrak{n}}^- \equiv{\mathfrak{n}}^-(H)
\]
to be the sum of the eigenspaces
 with positive, zero, and negative eigenvalues,
 respectively.  
Then 
$
{\mathfrak{p}}(H):={\mathfrak{l}}(H)
                   +{\mathfrak{n}}^+(H)
$
 is a parabolic subalgebra
 of ${\mathfrak{g}}$.  
Let ${\mathfrak {g}}_{\mathbb{R}}$ be the Lie algebra
 of a real reductive Lie group $G$, 
 and ${\mathfrak {g}}:={\mathfrak {g}}_{\mathbb{R}} \otimes
_{\mathbb{R}}{\mathbb{C}}$.  
If $H$ is a hyperbolic element
 of ${\mathfrak{g}}_{\mathbb{R}}$, 
 then 
\[
     P(H):=\{g \in G:
\operatorname{Ad}(g) {\mathfrak{p}}(H)
={\mathfrak{p}}(H)\}
 =L(H) \exp ({\mathfrak {n}}^+(H))
\]
 is a parabolic subgroup
 of $G$, 
 and ${\mathfrak{p}}_{\mathbb{R}}(H):=
 {\mathfrak{p}}(H) \cap {\mathfrak{g}}_{\mathbb{R}}$
 is its Lie algebra.  

We define the following subset 
 of ${\mathfrak{g}}_{\mathbb{R}}$
 by 
\begin{equation}
\label{eqn:reg}
{\mathfrak {g}}_{\mathbb{R}}^{\operatorname{reg, hyp}}
:=
\{H \in {\mathfrak {g}}_{\mathbb{R}}
:
H \text{ is hyperbolic, 
 and $L(H)$ is amenable} \}.  
\end{equation}
For a hyperbolic element $H$ in ${\mathfrak {g}}_{\mathbb{R}}$,
 ${\mathfrak {p}}_{\mathbb{R}}(H)$ is
 a minimal parabolic subalgebra of ${\mathfrak {g}}_{\mathbb{R}}$
 if and only if
 $H \in {\mathfrak {g}}_{\mathbb{R}}^{\operatorname{reg, hyp}}$
 by definition.
\begin{definition}
[${\mathfrak {g}}'$-compatible parabolic subalgebra]
\label{def:2.6}
We say a parabolic subalgebra ${\mathfrak{p}}$ 
 of ${\mathfrak{g}}$ is ${\mathfrak{g}}'$-{\it{compatible}}
 if there exists a hyperbolic element $H$
 in ${\mathfrak{g}}'$
 such that ${\mathfrak{p}} \equiv {\mathfrak{p}}(H)$.   
We say $P$ is $G'$-compatible
 if we can take $H$ in ${\mathfrak {g}}_{\mathbb{R}}'$.  
\end{definition}
If $P$ 
 is $G'$-compatible,
 then ${\mathfrak{p}}':={\mathfrak{p}} \cap {\mathfrak{g}}'$
 becomes a parabolic subalgebra
 of ${\mathfrak{g}}'$
 with Levi decomposition 
\[
{\mathfrak{p}}'
=
{\mathfrak {l}}'+({\mathfrak {n}}^+)'
\equiv
({\mathfrak{l}} \cap {\mathfrak{g}}')
 +({\mathfrak{n}}^+ \cap {\mathfrak{g}}')
\]
 and $P':= P \cap G'$
 becomes a parabolic subgroup of $G'$.  
The ${\mathfrak {g}}'$-compatibility
 is a sufficient condition 
 for the \lq{discrete decomposability}\rq\
 of generalized Verma modules
 $U({\mathfrak {g}}) \otimes_{U({\mathfrak {p}})} F$
 when restricted to the reductive subalgebra ${\mathfrak {g}}'$
 (see \cite{K12}).  

\vskip 1pc
By abuse of notation,
 we shall write also ${\mathcal{L}}_{\lambda} \to X$
 for the line bundle
 associated to $(\lambda, V)$
 instead of the previous notation
 ${\mathcal{V}} \to X$
 if $V$ is one-dimensional.  
We write $\lambda \gg 0$
 if $\langle d \lambda, \alpha \rangle \gg 0$
 for all $\alpha \in \Delta({\mathfrak {n}}^+, {\mathfrak {j}})$
 where ${\mathfrak {j}}$ is a Cartan subalgebra
 of ${\mathfrak {l}}$.  
\begin{theorem}
[local operators]
\label{thm:finVerma}
Let ${\mathfrak{g}}'$ be a reductive subalgebra 
 of ${\mathfrak{g}}$.  
Suppose ${\mathfrak{p}}={\mathfrak{l}} + {\mathfrak{n}}^+$
 is ${\mathfrak{g}}'$-compatible.  
\par\noindent
{\rm{1)}}\enspace
{\rm{(finite multiplicity)}}
\enspace
For any finite-dimensional representations $V$
 and $W$ of the parabolic subgroups $P$ and $P'$, 
 respectively,
 we have
\[
\dim \operatorname{Diff}_{G'}(C^{\infty}(X, {\mathcal{V}}), 
C^{\infty}(Y, {\mathcal{W}}))
 < \infty.  
\]
\par\noindent
{\rm{2)}}\enspace
{\rm{(uniformly bounded multiplicity)}}\enspace
If $({\mathfrak{g}},{\mathfrak{g}}')$ is a reductive symmetric pair
 and ${\mathfrak{n}}^+$ is abelian,
 then 
\[
\sup_W \dim 
\operatorname{Diff}_{G'}(C^{\infty}(X,{\mathcal{L}}_{\lambda}),
 C^{\infty}(Y,{\mathcal{W}}))
 =1, 
\]
for any one-dimensional representation 
${\mathbb{C}}_{\lambda}$ of $P$
 with $\lambda \gg 0$.  
Here $W$ runs over all finite-dimensional irreducible 
 representations of $P'$.  
\end{theorem}

\begin{proof}
1)\enspace
The classical duality
 between Verma modules
 and principal series representations
 (e.g. \cite{xhaja})
 can be extended 
 to the context
 of the restriction
 for reductive groups $G\downarrow G'$, 
 and the following bijection holds
 (see \cite[Corollary 2.9]{KP}):
\[
\operatorname{Hom}_{({\mathfrak{g}}',P')}
  (U({\mathfrak{g}}') \otimes_{U({\mathfrak{p}}')}W^{\vee},
   U({\mathfrak{g}}) \otimes_{U({\mathfrak{p}})}V^{\vee})
\simeq
\operatorname{Diff}_{G'}(C^{\infty}(G/P, {\mathcal{V}}), 
C^{\infty}(G'/P', {\mathcal{W}})).  
\]
Here $(\lambda^{\vee}, V^{\vee})$
 denotes the contragredient representation 
 of $(\lambda,V)$.  
Then the proof
 of Theorem \ref{thm:finVerma}
 is reduced to the next proposition.  
\end{proof}

\begin{proposition}
Let ${\mathfrak {g}}'$ be a reductive subalgebra
 of ${\mathfrak {g}}$.  
Suppose that ${\mathfrak {p}}={\mathfrak {l}}+{\mathfrak {n}}^+$
 is ${\mathfrak {g}}'$-compatible.  
\par\noindent
{\rm{1)}}\enspace
For any finite-dimensional ${\mathfrak{p}}$-module $F$ 
 and ${\mathfrak{p}}'$-module $F'$, 
\[
  \dim \operatorname{Hom}_{{\mathfrak{g}}'}
  (U({\mathfrak{g}}') \otimes_{U({\mathfrak{p}}')}F',
   U({\mathfrak{g}}) \otimes_{U({\mathfrak{p}})}F)
  <\infty.  
\] 
\par\noindent
{\rm{2)}}\enspace
If $({\mathfrak {g}}, {\mathfrak {g}}')$ is a symmetric pair
 and ${\mathfrak {n}}^+$ is abelian,
 then 
\[
\sup_{F'} \dim 
\operatorname{Hom}_{{\mathfrak {g}}'}
(U({\mathfrak {g}}') \otimes_{U({\mathfrak {p}}')}F',
U({\mathfrak {g}}) \otimes_{U({\mathfrak {p}})} {\mathbb{C}}_{\lambda})
=1
\]
for any one-dimensional representation ${\mathbb{C}}_{\lambda}$ 
 of ${\mathfrak {p}}$
 with $\lambda \ll 0$.  
Here the supremum is taken over all finite-dimensional
 simple ${\mathfrak {p}}'$-modules $F'$.  
\end{proposition}
\begin{proof}
1)\enspace
The proof is parallel to \cite[Theorem 3.10]{K12}
which treated the case
 where $F$ and $F'$ are simple modules
 of $P$ and $P'$, 
respectively.  
\par\noindent
2)\enspace
See \cite[Theorem 5.1]{K12}.  
\end{proof}
Hence Theorem \ref{thm:finVerma} is shown.  
We refer also to \cite[Theorem B]{K08}
 for an analogous statement
 to Theorem \ref{thm:finVerma} (2)
 which was formulated
 in the context of unitary representations.  

\begin{proof}
[Proof of Theorem \ref{thm:fmdiff}]
If $H$ is a generic hyperbolic element $H$
 in ${\mathfrak {g}}_{\mathbb{R}}'$, 
 then ${\mathfrak {p}}_{\mathbb{R}}(H)$
 is a minimal parabolic subalgebra of 
 ${\mathfrak {g}}_{\mathbb{R}}$
 owing to the rank assumption.  
Then Theorem \ref{thm:fmdiff} follows from Theorem \ref{thm:finVerma}.  
\end{proof}
\begin{remark}
In most cases
 where $G'$ is a proper non-compact subgroup 
 of $G$,
 for each representation $(\lambda,V)$ of $P$, 
 there exist continuously many representations
 $(\nu,W)$ of $P'$ 
 such that 
\[
\operatorname{Hom}_{G'}(C^{\infty}(X,{\mathcal{V}}),C^{\infty}(
Y,{\mathcal{W}}) )\ne \{0\}.  
\]
However,
 there are a few exceptional cases
 where only a finite number
 of irreducible representations 
$(\nu, W)$ of $P'$
 satisfy 
\[
\operatorname{Hom}_{G'}
(C^{\infty}(X,{\mathcal{V}}),
C^{\infty}(Y,{\mathcal{W}}))\ne \{0\}
\]
(\cite[Theorem 3.8]{xkzuckerman}).  
This happens 
 when $X=Y$ and $G \supsetneqq G'$.  
Even in this case,
Project A brought us a new interaction
 of the classical analysis
  (e.g. the Weyl operator calculus)
 and representation theory
 (e.g. infinite-dimensional representations
of minimal Gelfand--Kirillov dimensions)
 via symmetry breaking operators, 
 see \cite{xclare, KOP}
 for $(G,G')=
(GL(2n,{\mathbb{F}}),Sp(n,{\mathbb{F}}))$, 
 ${\mathbb{F}}={\mathbb{R}},{\mathbb{C}}$.  
\end{remark}

\section{F-method for continuous operators}
\label{sec:Fmethod}

The \lq{F-method}\rq\
 is a powerful tool to find singular vectors
 explicitly 
 in the Verma modules
 by using the algebraic Fourier transform
 \cite{xkhelgason}.  
We applied the \lq{F-method}\rq\
 in the previous papers
 \cite{KOSS,KP}
 to construct new covariant differential operators
 including classical Rankin--Cohen's 
 bi-differential operator
 \cite{C75,CMZ,xrankin}
 and Juhl's conformally covariant differential operator
 \cite{Juhl}.  

In this section,
 we generalize the idea 
 of the F-method
{}from local to non-local operators. 
We follow the notation of \cite{KP}.  
In particular,
 we regard distributions
 as generalized functions
 {\`a} la Gelfand
 rather than continuous linear forms
 on $C_c^{\infty}(X)$.  
Concrete examples
 will be discussed in Section \ref{sec:4}.  

We retain the setting 
 as before.  
Let $G_{\mathbb{C}}$ be a complexification
 of $G$.  
According to the Gelfand--Naimark decomposition 
$\mathfrak g=\mathfrak n^-+\mathfrak l+\mathfrak n^+$
of the complex reductive Lie algebra $\mathfrak{g}$, 
we have a diffeomorphism
$$
\mathfrak n^-\times L_{\mathbb{C}}\times \mathfrak n^+\to G_{\mathbb{C}},\quad
(X,\ell, Y)\mapsto (\exp X)\ell(\exp Y),
$$
onto an open subset $G^{\mathrm{reg}}_{\mathbb{C}}$
 containing the identity 
 of the complex Lie group $G_{\mathbb{C}}$. 
Let
$$
p_\pm: G^{\mathrm{reg}}_{\mathbb{C}}\longrightarrow
\mathfrak n^\pm,\qquad p_o:G^{\mathrm{reg}}_{\mathbb{C}}\to L_{\mathbb{C}},
$$ be the projections characterized by the identity
\[
\exp(p_-(g)) \, p_o(g) \, \exp(p_+(g))=g.
\]
Then the definition 
 of the following maps $\alpha$ and $\beta$
 is independent
 of the choice of $G_{\mathbb{C}}$:  
\begin{equation}
\label{eqn:alpha}
(\alpha,\beta)
:\mathfrak g \times \mathfrak n^-\to\mathfrak l \oplus {\mathfrak {n}}^-,
\quad
 (D,C)\mapsto \left.\frac {d}{dt}\right\vert_{t=0}
\left(p_o\left(e^{tD}e^C\right),p_-\left(e^{tD}e^C\right)\right).  
\end{equation}
For example,
 if ${\mathfrak{n}}^{\pm}$ are abelian
 or equivalently
 if $({\mathfrak {g}},{\mathfrak {l}})$ is 
 a symmetric pair, 
 then $\alpha:{\mathfrak {g}} \times {\mathfrak {n}}^-
\to {\mathfrak {l}}$
 takes the form
\[
  \alpha(D,C)=
\begin{cases}
[D,C]
  \qquad
  &\text{for }\,\,
  C \in {\mathfrak{n}}^-, D \in {\mathfrak{n}}^+,  
\\
0
  \qquad
  &\text{for }\,\,
  C \in {\mathfrak {n}}^-, 
  D \in {\mathfrak {n}}^- + {\mathfrak {l}}.  
\end{cases}
\]
For $D \in {\mathfrak {g}}$, 
 $\beta(D,\cdot\,)$ induces a linear map ${\mathfrak {n}}^-\to{\mathfrak {n}}^-$, 
 and thus we may regard $\beta(D, \cdot)$
 as a holomorphic vector field
 on $\mathfrak n^-$
 via the identification $\mathfrak n^-\ni C\mapsto\beta(D,C)
\in\mathfrak n^-\simeq T_C \mathfrak n^-.$

We recall 
${\mathfrak {n}}_{\mathbb{R}}^{\pm}$
 are real forms 
 of the complex Lie algebras ${\mathfrak {n}}^{\pm}$.  
Let $N^{\pm}:=\exp ({\mathfrak {n}}_{\mathbb{R}}^{\pm})$.  
Then 
\begin{equation}
\label{eqn:Bruhat}
\iota:{\mathfrak{n}}_{\mathbb{R}}^- \to G/P, 
\quad
Z \mapsto (\exp Z)P
\end{equation}
defines an open Bruhat cell $N^-P/P$
 in the real flag variety $G/P$.

We denote by ${\mathbb{C}}_{2 \rho}$
 the one-dimensional representation 
 of $P$
 on $| \Lambda^{\dim {\mathfrak {n}}}(\mathfrak {n})|$.  
The infinitesimal representation
 will be denoted by $2 \rho$.

Let $(\lambda, V)$ be a finite-dimensional representation of $P$
 with trivial action of $N^+$.  
Then the dualizing bundle 
 ${\mathcal{V}}^{\ast} := {\mathcal{V}}^{\vee} \otimes \Omega_{G/P}$
is given by 
\[
{\mathcal{V}}^{\ast}\simeq G\times_P
 (V^{\vee}
 \otimes {\mathbb{C}}_{2\rho})
\]
 as a homogeneous vector bundle.  
The pull-back of ${\mathcal{V}}^{\ast} \to G/P$
 to ${\mathfrak {n}}_{\mathbb{R}}^-$ via \eqref{eqn:Bruhat}
 defines the trivial vector bundle
 ${\mathfrak {n}}_{\mathbb{R}}^- \times V^{\vee} \to {\mathfrak {n}}_{\mathbb{R}}^-$,
 and induces an injective morphism
$
\iota^{\ast}:
C^{\infty}(G/P, {\mathcal{V}}^{\ast})
 \to C^{\infty}({\mathfrak {n}}_{\mathbb{R}}^-)\otimes V^{\ast}.
$  
The infinitesimal representation of the regular
 representation $C^{\infty}(G/P, {\mathcal{V}}^{\ast})$
 is given as an operator 
 on $C^{\infty}({\mathfrak {n}}_{\mathbb{R}}^-)\otimes V^{\vee}$
 by 
\begin{equation}
\label{eqn:pil}
d \pi_{\lambda}^{\ast}(D)
:=\langle d \lambda^{\vee} + 2 \rho \operatorname{id}_{V^{\vee}}
, 
\alpha(D, \cdot)\rangle
 - \beta(D, \cdot) \otimes \operatorname{id}_{V^{\vee}}
\quad
\text{for }
D \in {\mathfrak {g}}.  
\end{equation}
We may regard $d \pi_{\lambda}^{\ast}(D)$
 as an $\operatorname{End}(V^{\vee})$-valued
 holomorphic differential operator
 on ${\mathfrak {n}}^-={\mathfrak {n}}_{\mathbb{R}}^-
\otimes_{\mathbb{R}} \mathbb{C}$.

\vskip 1pc
Any continuous operator
 $T:C^{\infty}(X,{\mathcal{V}}) \to C^{\infty}(Y,{\mathcal{W}})$
 is given by a distribution kernel 
$K_T \in {\mathcal{D}}'(X \times Y,{\mathcal{V}}^{\ast} \boxtimes 
{\mathcal{W}})$
 by the Schwartz kernel theorem.  
We write 
\[
 m: G \times G' \to G, 
\quad
 (g,g') \mapsto (g')^{-1}g,  
\]
for the multiplication map.  
Then the pull-back $m^{\ast} K_T$
 is regarded as an element
 of ${\mathcal{D}}'(X, {\mathcal {V}}^{\ast}) \otimes W$.  
We have from \cite{KS} 
 the following two propositions:  
\begin{proposition}
\label{prop:KT}
The correspondence
 $T \mapsto m^{\ast} K_T$ induces a bijection:
\[
  \operatorname{Hom}_{G'}(C^{\infty}(X,{\mathcal{V}}),C^{\infty}(Y,{\mathcal{W}}))
\overset{\sim}\to
({\mathcal{D}}'(X,{\mathcal{V}}^{\ast})\otimes W)^{\Delta(P')}.  
\]
\end{proposition}

\begin{proposition}
\label{prop:KTN}
Assume
 that the natural multiplication map
 $P' \times N^- \times P \to G$
 is surjective,
 namely, 
\begin{equation}
\label{eqn:PNPG}
P' N^- P =G.  
\end{equation}
Then $\iota^{\ast}m^{\ast}K_T$ is a $W$-valued tempered
 distribution on ${\mathfrak {n}}_{\mathbb{R}}^-$, 
 and the correspondence
 $T \mapsto \iota^{\ast}m^{\ast}K_T$
 is injective:
\begin{equation}
\label{eqn:KTN}
  \operatorname{Hom}_{G'}(C^{\infty}(X,{\mathcal{V}}),C^{\infty}(Y,{\mathcal{W}}))
\hookrightarrow
{\mathcal{S}}'({\mathfrak {n}}_{\mathbb{R}}^-)\otimes W.  
\end{equation}
\end{proposition}
The idea of the F-method
 for non-local operators
 is to characterize the image
 of \eqref{eqn:KTN}
 by the Fourier transform
 ${\mathcal{S}}'({\mathfrak {n}}_{\mathbb{R}}^-)
 \overset \sim \to 
 {\mathcal{S}}'({\mathfrak {n}}_{\mathbb{R}}^+)$.

For this, 
 we recall the algebraic Fourier transform
 of the Weyl algebra.  
Let $E$ be a complex vector space, 
 and denote by ${\mathcal{D}}(E)$  
 the ring of holomorphic differential operators on $E$
with polynomial coefficients. 

\begin{definition}
[algebraic Fourier transform]
\label{def:hat}
 We define the \emph{algebraic Fourier transform} as an isomorphism
 of the Weyl algebras on $E$ and its dual space $E^\vee$:
$$
\mathcal D(E)\to\mathcal D(E^\vee), \qquad S\mapsto \widehat S,
$$
induced by
$$\widehat{\frac\partial{\partial z_j}}:=-\zeta_j,\quad
\widehat z_j:=\frac\partial{\partial\zeta_j},\quad 1\leq j\leq n=\dim E, 
$$ 
where $(z_1,\ldots,z_n)$ are coordinates on $E$
 and $(\zeta_1,\ldots,\zeta_n)$ are the dual coordinates on $E^\vee$.  
The definition does not depend on the choice of coordinates. 
\end{definition}

Suppose $P$ is a $G'$-compatible parabolic subgroup
 of $G$ (Definition \ref{def:2.6}),
 and we take $P'$ to be a parabolic subgroup 
 of $G'$ defined by 
 $P \cap G'=L'\exp({\mathfrak {n}}_{\mathbb{R}}^+)'$.  
We define a subspace
 of $\operatorname{Hom}(V,W)$-valued tempered distributions
 on ${\mathfrak {n}}_{\mathbb{R}}^+$
 by 
\[
\Sol(V, W)^{\wedge}
:=\{
  F (\xi)\in {\mathcal{S}}'({\mathfrak {n}}_{\mathbb{R}}^+)
\otimes \operatorname{Hom}(V,W)
:
F \text{ satisfies \eqref{eqn:invL} and \eqref{eqn:invn}
on $ {\mathfrak {n}}_{\mathbb{R}}^+$}
\}, 
\]
where 
\begin{equation}
\label{eqn:invL}
   \nu(l) \circ F(\operatorname{Ad}(l^{-1})\cdot)
   \circ \lambda(l^{-1})
   =
   F(\cdot)
   \quad\text{for all }
   l \in L', 
\end{equation}
\begin{equation}
\label{eqn:invn}
(\widehat{d\pi_{\lambda}^{\ast}(C)}\otimes \operatorname{id}_W
+\operatorname{id}_{V^{\vee}} \otimes d \nu(C))
|_{\zeta =-i \xi}
\,\,
F
=0
\quad
\text{for all }C \in ({\mathfrak {n}}^+)'.  
\end{equation}

Combining \cite{KP} and Proposition \ref{prop:KTN}, 
 we obtain:
\begin{theorem}
[F-method for continuous operators]
\label{thm:Fcont}
Let $G$ be a reductive linear Lie group,
 and $G'$ a reductive subgroup.  
Suppose $P$ is a $G'$-compatible parabolic subgroup of $G$, 
 and $P'=P \cap G'$.  
We assume \eqref{eqn:PNPG}.  

Then the Fourier transform ${\mathcal{F}}_{\mathbb{R}}$
 (see Remark \ref{rem:45})
 of the distribution kernel 
 induces the following bijection:
\begin{equation}
\label{eqn:FC}
  \operatorname{Hom}_{G'}(C^{\infty}(G/P, {\mathcal{V}}), 
C^{\infty}(G'/P', {\mathcal{W}}))
  \overset \sim \to 
  \Sol(V,W)^{\wedge}. 
\end{equation}
\end{theorem}
\begin{remark}
\label{rem:Fcont}
Theorem \ref{thm:Fcont} extends 
 the following bijection:
\begin{equation}
\label{eqn:FD}
  \operatorname{Diff}_{G'}(C^{\infty}(X, {\mathcal{V}}), 
C^{\infty}(Y, {\mathcal{W}}))
  \simeq
  \Sol(V,W)^{\wedge} \cap \operatorname{Pol}({\mathfrak {n}}^+), 
\end{equation}
which was proved in \cite{KP}
 (cf. \cite{KOSS})
 without the assumption \eqref{eqn:PNPG}.  
\end{remark}
\begin{remark}\label{rem:45}
Suppose $E_{\mathbb{R}}$ is a real form of $E$.  
If $f$ is a compactly supported distribution
 on $E_{\mathbb{R}}$, 
 the Fourier transform ${\mathcal{F}}_{\mathbb{R}}f$
 extends holomorphically 
 on the entire complex vector space $E^{\vee}$.  
In this case, 
 we may compare two conventions
 of \lq{Fourier transforms}\rq\
\begin{align*}
{\mathcal{F}}_c
:\,\,&
{\mathcal{E}}'(E_{\mathbb{R}}) \to {\mathcal{O}}(E^{\vee}), 
\quad
f \mapsto \int_{E_{\mathbb{R}}} f(x) e^{\langle x, \zeta \rangle} dx, 
\\
{\mathcal{F}}_{\mathbb{R}}
:\,\,&
{\mathcal{E}}'(E_{\mathbb{R}}) \to {\mathcal{O}}(E^{\vee}), 
\quad
f \mapsto \int_{E_{\mathbb{R}}} f(x) e^{-i\langle x, \xi \rangle} dx, 
\end{align*}
where ${\mathcal{E}}'(E_{\mathbb{R}})$
 stands for the space
 of compactly supported distributions.  
We note
\begin{equation}
\label{eqn:FRc}
({\mathcal{F}}_{\mathbb{R}}f)(\xi)
=
({\mathcal{F}}_c f)(\zeta)
\qquad
\text{with $\zeta=-i \xi$}.  
\end{equation}
In \cite{KP}
 we have adopted ${\mathcal{F}}_c$
 instead of ${\mathcal{F}}_{\mathbb{R}}$.  
An advantage 
 of ${\mathcal{F}}_c$ is
 that the algebraic Fourier transform defined in Definition \ref{def:hat} satisfies
\begin{equation}\label{eqn:F-trans}
 \widehat T=\mathcal F_{c}\circ T\circ \mathcal F_{c}^{-1}
\quad\text{for all $T \in \mathcal{D}(E)$}, 
\end{equation}
which simplifies actual computations 
 in the F-method.  
In the case 
where ${\mathfrak {n}}^+$ is abelian, 
the bijection \eqref{eqn:FD}
 is compatible
 with the symbol map
\[
  \operatorname{Diff}_{G'}(C^{\infty}(X, {\mathcal{V}}), 
C^{\infty}(Y, {\mathcal{W}}))
  \to
  \operatorname{Pol}({\mathfrak {n}}^+)
  \otimes 
  \operatorname{Hom}_{\mathbb{C}}(V,W), 
\]
if we use ${\mathcal{F}}_c$
 instead of ${\mathcal{F}}_{\mathbb{R}}$
 (see \cite[Theorem 3.5]{xkhelgason} or \cite{KP}).  
\end{remark}

\section{Conformally covariant symmetry breaking}
\label{sec:3}
In this section
 we set up some notation 
 for conformally covariant operators
in the setting
 where the groups $G' \subset G$
 act conformally
 on two pseudo-Riemannian manifolds
 $Y \subset X$, 
 respectively.

Let $X$ be a smooth manifold
 equipped with a pseudo-Riemannian structure $g$.  
Suppose a group $G$ acts conformally on $X$.  
The action will be denoted 
by $
  L_h : X \to X
$, 
$
  x \mapsto L_h x  
$
 for $h \in G$.  
Then there exists a positive-valued function 
 $\Omega$ on $G \times X$
 such that 
\[
  L_h^{\ast}(g_{L_h x}) = \Omega(h,x)^2 g_x
  \quad
  \text{ for any }
  h \in G, \text{ and } x \in X.  
\]

Fix $\lambda \in {\mathbb{C}}$, 
 and we define a linear map
 $\varpi_{\lambda}(h^{-1}):C^{\infty}(X) \to C^{\infty}(X)$
 by 
\[
  (\varpi_{\lambda}(h^{-1})f)(x)
  := \Omega(h,x)^{\lambda}f(L_hx).  
\]
Since the conformal factor $\Omega$
 satisfies the cocycle condition:
\[
\Omega(h_{1}h_{2},x)= \Omega(h_{1},L_{h_{2}}x)\Omega(h_{2},x)
\quad
 \text{for }h_1, h_2 \in G, x\in X, 
\]
 we have formed a family
 of representations
$\varpi_{\lambda}
\equiv \varpi_{\lambda}^X$
of $G$
 on $C^{\infty}(X)$
 with complex parameter $\lambda$
 (see \cite[Part I]{xkors} for details).  
\begin{remark}
{{\rm{1)}}} \enspace
If $G$ acts on $X$ as isometries,
 then $\Omega \equiv 1$ and therefore the representation
 $\varpi_{\lambda}$ does not depend on $\lambda$.  
\par\noindent
{{\rm{2)}}} \enspace
Let $n$ be the dimension of $X$.  
Then in our normalization,
 $(\varpi_n, C^{\infty}(X))$
 is isomorphic to the representation 
 on $C^{\infty}(X, \Omega_X)$
 where $\Omega_X$ denotes the bundle of volume densities.  
\end{remark}

Let $\operatorname{Conf}(X)$ be 
 the full group of conformal transformations 
 on $(X,g)$.  
Suppose $Y$ is a submanifold of $X$
 such that the restriction $g|_Y$ is non-degenerate.  
Clearly,
 this assumption is automatically satisfied
 if $(X,g)$ is a Riemannian manifold.  
We define a subgroup of $\operatorname{Conf}(X)$
 by 
\[
  \operatorname{Conf}(X;Y)
  =
  \{\varphi \in \operatorname{Conf}(X)
   :
   \varphi(Y) \subset Y\}.  
\]

For $\varphi \in \operatorname{Conf}(X;Y)$, 
 $\varphi$ induces a conformal transformation on $(Y,g|_Y)$, 
 and we get a natural group homomorphism
\[
   \operatorname{Conf}(X;Y)
   \to 
   \operatorname{Conf}(Y).  
\]
We write $ \operatorname{Conf}_Y(X)$ for its image.

Thus, 
 for $\lambda, \nu \in {\mathbb{C}}$, 
 we have the following two representations:
\begin{align*}
 \varpi_{\lambda}^X:& \operatorname{Conf}(X) \to GL_{\mathbb{C}}(C^{\infty}(X)), 
\\
 \varpi_{\nu}^Y:& \operatorname{Conf}(X;Y) \to GL_{\mathbb{C}}(C^{\infty}(Y)).  
\end{align*}
We are ready to state the following problem:
\begin{problem}
\label{prob:3.2}
{\rm{1)}} \enspace
Classify $(\lambda, \nu) \in {\mathbb{C}}^2$
 such that there exists a non-zero continuous/differential
 operator 
\[
  T_{\lambda, \nu}:C^{\infty}(X) \to C^{\infty}(Y)
\]
satisfying 
\[
\varpi_{\nu}^Y (h) \circ T_{\lambda, \nu}
  = T_{\lambda, \nu} \circ \varpi_{\lambda}^X(h)
\quad 
 \text{for all }
 h \in \operatorname{Conf}(X; Y).  
\]
\par\noindent
{\rm{2)}} \enspace
Find explicit formulas
 of the operators $T_{\lambda, \nu}$.  
\end{problem}

We begin with an obvious example.  
\begin{example}
\label{ex:1}
Suppose $\lambda = \nu$.  
We take $T_{\lambda, \nu}$ 
 to be the restriction of functions from $X$ to $Y$.  
Clearly,
 $T_{\lambda, \nu}$ intertwines $\varpi_{\lambda}$
 and $\varpi_{\nu}$.  
\end{example}
Problem \ref{prob:3.2} concerns a geometric aspect
 of the general branching problem
 for representations
 with respect to the restriction 
 $G \downarrow G'$
 in the case where 
\[
(G,G')=(\operatorname{Conf}(X), \operatorname{Conf}_Y(X)).  
\]
We shall see in Section \ref{sec:4}
 that a special case of Problem \ref{prob:3.2}
 is a special case of Project A.  
In Section \ref{sec:4}, 
 we shall write $\varpi_{\lambda}^G$ and $\varpi_{\nu}^{G'}$
 for $\varpi_{\lambda}\equiv \varpi_{\lambda}^X$
 and $\varpi_{\nu}\equiv \varpi_{\nu}^Y$,
respectively,
 in order to emphasize the groups $G$ and $G'$.

We continue 
with some further examples 
of Problem \ref{prob:3.2}.  

\begin{example}
[Eastwood--Graham]
\label{ex:GGV}
Let $X=Y$ be the sphere $S^n$ endowed with a standard Riemannian metric,
 and we take $G=G'$ to be the Lorentz group $SO(n+1,1)$
 that act conformally on $S^n$ by the M\"obius transform.
In this case, $G$ is a semisimple Lie group of real rank one,
 and all local/non-local intertwining operators
 were classified by Gelfand--Graev--Vielenkin \cite{GGV} for $n=1,2$,
 by Eastwood--Graham \cite{EGR} for general $n$ when
 $\mathcal V$ and $\mathcal W$ are line bundles.
In particular,
 all conformal invariants for densities in this case
 are given by residues of continuous conformal intertwiners.  
It is noted
 that an analogous statement 
 is not always true
 when $G'\ne G$
 (cf. Remark \ref{rem:Leven}). 
\end{example}
\begin{example}
\label{ex:2}
Suppose $X=Y$.  
Then $G=G'$.  
Let $n$ be the dimension of the manifold,
 and we consider the following specific parameter:
\[\lambda = \frac 1 2 n -1,
\quad
\nu=\frac 1 2 n +1.  \]
Then the Yamabe operator
 $\widetilde{\Delta_X}$ satisfies
$
  \varpi_{\nu} \circ \widetilde{\Delta_X}
=
\widetilde{\Delta_X} \circ \varpi_{\lambda}
$, 
where 
\[
  \widetilde{\Delta_X}
  :=
   \Delta_X-\frac{n-2}{4(n-1)}\kappa.  
\]
Here $\Delta_X$ is the Laplacian
 for the pseudo-Riemannian manifold $(X,g)$,
 and $\kappa$ is the scalar curvature.

In particular,
 if $X$ is the direct product
 of two spheres $S^p \times S^q$
 endowed with the pseudo-Riemannian structure
 $g_{S^p}\oplus (-g_{S^q})$
 of signature $(p,q)$,
then the kernel $\operatorname{Ker}(\widetilde{\Delta_X})$
 gives rise to an important irreducible unitary representation
 of $\operatorname{Conf}(X) \simeq O(p+1,q+1)$, 
 so-called a {\it{minimal representation}}
 for $p+q$ even, 
 $p,q \ge 1$, 
 and $p+q \ge 6$
\cite{xkcheck, xkors}
 in the sense
 that its annihilator in the enveloping algebra
 $U({\mathfrak {o}}(p+q+2, {\mathbb{C}}))$
 is the Joseph ideal
 \cite{xbz,xkostant}.  
The same representation is known 
 to have different realizations
 and constructions,
 e.g. the local theta correspondence
 \cite{xhuzh}, 
 and the Schr{\"o}dinger model 
 \cite{xkmanoAMS}.  
\end{example}

\begin{example}
[\cite{Juhl, KOSS, KS}]
\label{ex:3}
Let $X$ be the standard sphere $S^n$
 and $Y=S^{n-1}$
 a totally geodesic hypersurface
 (\lq{great circle}\rq\ when $n=2$).  
Then we have covering maps
\begin{alignat*}{3}
   G&:=O(n+1,1)&& \twoheadrightarrow  &&\operatorname{Conf}(X), 
\\
    &\hphantom{mmmm}\cup &&     && \hphantom{mm}\cup
\\
   G'&:=O(n,1)\hphantom{ii}&&  \twoheadrightarrow &&\operatorname{Conf}_Y(X).  
\end{alignat*}
Then non-zero $G'$-equivariant
 differential operators
 $T_{\lambda, \nu}:C^{\infty}(X) \to C^{\infty}(Y)$ exist
 if and only if 
 the parameter $(\lambda,\nu)$ satisfies 
 $\nu-\lambda \in \{0,2,4,\cdots\}$.  
(Here, 
 the parity condition arises from the fact
 that $G$ and $G'$ are disconnected groups, 
 cf. \cite{Juhl, KOSS}.)
In this case
$
  \dim \operatorname{Diff}_{G'}(C^{\infty}(X),C^{\infty}(Y))=1.  
$

In order to describe this differential operator $T_{\lambda,\nu}$
 explicitly,
 we use the stereographic projection 
\begin{equation}
\label{eqn:SR}
  S^n \to {\mathbb{R}}^n \cup \{\infty\},
  \quad
  (s, \sqrt{1-s^2}\omega) \mapsto \sqrt{\frac{1-s}{1+s}}\omega,
\end{equation}
and the corresponding twisted pull-back
 (see \cite[Part I]{xkors})
for the conformal map \eqref{eqn:SR}, 
\begin{equation}
\label{eqn:il}
  \iota_{\lambda}^{\ast}
  :
   C^{\infty}(S^n) \hookrightarrow C^{\infty}({\mathbb{R}}^n),
   \quad
   f \mapsto F
\end{equation}
is given by 
\[
  F(r \omega) 
:=(1+r^2)^{-\lambda}f(\frac{1-r^2}{1+r^2},\frac{2r}{1+r^2}\omega)
\quad
\text{for $r >0$ and $\omega \in S^{n-1}$}.
\]
In the coordinates,
 we realize the submanifold $Y$ correspondingly
 to the hyperplane $x_n=0$
 via \eqref{eqn:SR}, 
 namely,
 we have a commutative diagram
 of the stereographic projections:
\begin{alignat*}{3}
X=& \,S^n &&\longrightarrow \,\,\,&&{\mathbb{R}}^n \cup \{\infty \}
\\
 &  \,\,\cup && && \,\,\,\,\cup 
\\
Y=& \,S^{n-1} &&\longrightarrow \,\,\,&&{\mathbb{R}}^{n-1} \cup \{\infty \}
                  =\{(x_1, \cdots, x_{n-1},x_n)\in {\mathbb{R}}^n : x_n=0 \} \cup \{\infty \}.  
\end{alignat*} 
Accordingly,
 the subgroup $G'$ is defined
 as the isotropy subgroup
 of $G$
 at $e_n={}^{t\!}(0,\cdots, 0, 1,0) \in {\mathbb{R}}^{n+1}$.

Then for $\nu-\lambda=2l$
 ($l \in {\mathbb{N}}$), 
 the equivariant differential operator
 $T_{\lambda,\nu}$ is a scalar multiple
 of the following differential operator:
\begin{align}
 \tcC{\lambda}{\nu}
:\,\, &C^{\infty}({\mathbb{R}}^n) \to C^{\infty}({\mathbb{R}}^{n-1}),
\notag
\\
\label{eqn:Cop}
  &F \mapsto \sum_{j=0}^l 
\frac{2^{2l-2j}\prod_{i=1}^{l-j} \Bigl(\frac{\lambda+\nu-n-1}{2}+i\Bigr)}
{j!(2l-2j)!}
   \Delta_{\mathbb{R}^{n-1}}^j \Bigl(\frac{\partial}{\partial x_n}\Bigr)
^{2l-2j} F|_{x_n=0}.  
\end{align}
The differential operator $\tcC\lambda\nu$ can be written
by using the Gegenbauer polynomial as follows.  
Let $C_N^\mu(t)$ be the Gegenbauer polynomial
of degree $N$, 
 and we inflate $C_N^\mu(t)$ to a polynomial of two variables $v$, $t$
 by
\begin{align}
\widetilde{C}_{2l}^\mu (v,t)
:=& \frac{\Gamma(\mu)}{\Gamma(\mu+l)} v^l C_{2l}^\mu (\frac{t}{\sqrt{v}})
\label{eqn:Cmul2}
\\
=& \sum_{j=0}^l \frac{(-1)^j 2^{2l-2j}}{j!(2l-2j)!}
   \prod_{i=1}^{l-j} (\mu + l + i - 1) v^j t^{2l-2j} .
\notag
\end{align}
We note
 that the definition \eqref{eqn:Cmul2} makes sense
 if $v$, $t$ are elements
 in any commutative algebra $R$.  
In particular,
 taking $R={\mathbb{C}}[\frac{\partial}{\partial x_1}, 
\cdots, \frac{\partial}{\partial x_n}]$, 
we have the following expression:
\[
\tcC{\lambda}{\nu} F=
\widetilde{C}_{2l}^{\lambda-\frac{n-1}{2}}
   (-\Delta_{\mathbb{R}^{n-1}}, \frac{\partial}{\partial x_n})
F|_{x_n=0}.  
\]
A.~Juhl \cite{Juhl} proved the formula \eqref{eqn:Cop}
 by a considerably long computation based on recurrence relations.  
In \cite{KOSS}
 we have provided a new and simple proof
 by introducing another differential equation
 (see \eqref{eqn:fund} below)
 which controls the operators $\tcC{\lambda}{\nu}$
 ({\it{F-method}}).  
In the next section we give yet another proof
 of the formula \eqref{eqn:Cop}
{}from the residue calculations
 of non-local symmetry breaking operators
 with meromorphic parameter (see Theorem \ref{thm:AC}). 
\end{example}

\section{Conformally covariant non-local operators}
\label{sec:4}
In this section
 we analyze both non-local and local,
 conformally covariant operators
 in the setting of Example \ref{ex:3}
 by using the F-method.  

{}From the viewpoint 
 of representation theory,
 Example \ref{ex:3}
 deals with symmetry breaking differential operators
 between spherical principal series
 representations $\varpi_{\lambda}^G$
 and those $\varpi_{\nu}^{G'}$
 when $(G,G')=(O(n+1,1),O(n,1))$.  
In this case 
 all the assumptions
 for Theorems \ref{thm:fmdiff} and \ref{thm:fimHom}
 are fulfilled, 
 and therefore we tell a priori
 that both of the sides in \eqref{eqn:DC}
 are of uniformly bounded dimensions
 (actually,
 at most two-dimensional).  
In the joint work \cite{KS}
with B.~Speh, 
we give a complete classification
 of such symmetry breaking operators
 between line bundles for both non-local and local ones 
 with explicit generators,
which seems to be the first complete example
 of Project A
 in a setting where $G' \subsetneqq G$.

However,
 the techniques employed in \cite{KS} are not the F-method.  
Therefore
 it might be of interest
 to illuminate some of the key results of \cite{KS}
 {}from the scope of a generalized F-method
 (Section \ref{sec:Fmethod}).  
To achieve this aim 
 for the current section,
 we focus on the functional equations
(Theorem \ref{thm:TAATE})
 among non-local symmetry breaking operators
 $\tcA{\lambda}{\nu}$ 
(see \eqref{eqn:Atilde} below)
 and the relation between $\tcA{\lambda}{\nu}$
 and the differential operators
 $\tcC{\lambda}{\nu}$
 (Theorem \ref{thm:AC}).

\vskip 1pc
The novelty here
 is the following correspondence:
\vskip 1pc
{
\renewcommand{\arraystretch}{1.5}
\begin{tabular}{c|c}
   Symmetry breaking operators   &  F-method
\\
\hline
  Functional equations (Theorem \ref{thm:TAATE})
  & Kummer's relation
\\
  Residue formulae of $\tcA{\lambda}{\nu}$(Theorem \ref{thm:AC})
  &${}_2F_1 (a,b;c;z)$ reduces to a polynomial
    
\\
  & if $a \in -{\mathbb{N}}$
\end{tabular}
}

First of all we review quickly 
 some notation
 and results from \cite{KS}.  
Then the rest of this subsection
 will be devoted to provide
 some perspectives
{}from the F-method.

We set $|x|:=(x_1^2+\cdots + x_{n-1}^2)^{\frac 1 2}$
 for $x=(x_1, \cdots, x_{n-1}) \in {\mathbb{R}}^{n-1}$.  
If $(\lambda,\nu) \in {\mathbb{C}}^2$ satisfies
\begin{equation}
\label{eqn:lnopen}
\operatorname{Re}(\lambda-\nu)>0
\quad
\text{ and }
\quad
\operatorname{Re}(\lambda+\nu)>n-1,
\end{equation}
then 
\[
   \ka\lambda\nu(x,x_n):=|x_n|^{\lambda+\nu-n}(|x|^2+x_n^2)^{-\nu}
\]
 is locally integrable on ${\mathbb{R}}^n$,
 and the integral operator
\[
C_c^{\infty}({\mathbb{R}}^n) \to C^{\infty}({\mathbb{R}}^{n-1}),
\quad
F \mapsto 
\int_{{\mathbb{R}}^n} F(y,y_n) \ka\lambda\nu(x-y,-y_n)dy d y_n
\]
extends to a $G'$-intertwining operator
 via \eqref{eqn:il}
\[
\cA{\lambda}{\nu}:C^{\infty}(S^n) \to C^{\infty}(S^{n-1}),
\]
namely,
\[
\cA{\lambda}{\nu}\circ \varpi_{\lambda}^G(h)
=
\varpi_{\nu}^{G'}(h)\circ \cA{\lambda}{\nu}
\quad
\text{ for all $h \in G'$.}  
\]

The important property
 of our symmetry breaking operators $\cA{\lambda}{\nu}$
 is the existence
 of the meromorphic continuation
 to $(\lambda,\nu) \in {\mathbb{C}}^2$
 (see Theorem \ref{thm:meroA})
 and the functional equations
 satisfied by 
 $\cA {\lambda}{\nu}$
 and the Knapp--Stein intertwining operators
 (see Theorem \ref{thm:TAATE}).  
We note
 that the celebrated theorem \cite{xatiyah,xbege}
 on meromorphic continuation
 of distributions
 does not apply immediately to our distribution 
 $\ka\lambda\nu$
 because the two singularities $x_n=0$
 and $|x|^2+x_n^2=0$
 (i.e. the origin) 
 have an inclusive relation
 and are not transversal.  
Further,
 it is more involved 
 to find the location
 of the poles and their residues.  
In \cite{KS}
 we have found all the poles 
 and their residues explicitly,
 and in particular,
 we have the following theorems:
\begin{theorem}
\label{thm:meroA}
We normalize
\begin{equation}
\label{eqn:Atilde}
\tcA{\lambda}{\nu}:=
\frac{1}{\Gamma(\frac{\lambda+\nu-n+1}{2})
        \Gamma(\frac{\lambda-\nu}{2})}\cA{\lambda}{\nu}.  
\end{equation}
Then $\tcA{\lambda}{\nu}:C^{\infty}(S^n) \to C^{\infty}(S^{n-1})$, 
 initially holomorphic in the domain given by \eqref{eqn:lnopen},
 extends to a continuous operator
 for all $(\lambda,\nu)\in {\mathbb{C}}^2$
 and satisfies
\[
\tcA{\lambda}{\nu}\circ \varpi_{\lambda}^G(h)
=
\varpi_{\nu}^{G'}(h)\circ \tcA{\lambda}{\nu}
\quad
\text{for all }
h \in G'.  
\]
Further, 
 $\A_{{\lambda},{\nu}} f$ is a holomorphic function of 
 $(\lambda,\nu)$ on the entire space ${\mathbb{C}}^2$
 for all $f \in C^{\infty}(S^n)$.  
\end{theorem}

\begin{theorem}
[Residue formula]
\label{thm:AC}
If $\nu-\lambda =2l$
 with $l \in {\mathbb{N}}$
 then 
\[
\tcA{\lambda}{\nu}
=
\frac{(-1)^l l! \pi^{\frac{n-1}{2}}}{2^{2l} \Gamma(\nu)}
\tcC{\lambda}{\nu}.  
\]
\end{theorem}
\begin{remark}
\label{rem:Leven}
For $(\lambda,\nu)$ belonging to 
\[
L_{\operatorname{even}}
:=
\{(\lambda,\nu) \in {\mathbb{Z}}^2
: \lambda \le \nu \le 0,
 \lambda-\nu \equiv 0 \mod 2\}, 
\]
 the conformally covariant differential operator
 $\tcC \lambda \nu$ cannot be obtained
 as the residue
 of $\tcA \lambda\nu$.  
This discrete set $L_{\operatorname{even}}$
is exactly the zero-set of the symmetry breaking operators
$\tcA \lambda\nu$
and is the most interesting place
 of symmetry breaking
\cite{KS}.  
(We note that $L_{\operatorname{even}}$
 is of codimension two in ${\mathbb{C}}^2$!)
\end{remark}
The proof of Theorem \ref{thm:AC}
 in \cite{KS}
 is to use explicit formulae
 of the action
of $\tcA{\lambda}{\nu}$ and $\tcC{\lambda}{\nu}$
 on $K$-fixed vectors
 ({\it{spherical vectors}}).  
Instead,
we apply here the generalized F-method 
 and give an alternative proof
 of Theorem \ref{thm:AC},
 which is of more analytic nature
 and without using computations
 for specific $K$-types.

We write $\KA\lambda\nu$ and $\KC\lambda\nu$
 for the distribution kernels
 of the normalized symmetry breaking operator $\tcA\lambda\nu$
 and the conformally covariant differential operator
 $\tcC\lambda\nu$, 
respectively.  
For $(\lambda,\nu)$ belonging to the open domain 
\eqref{eqn:lnopen}, 
we have 
\begin{align*}
\KA{\lambda}{\nu} (x,x_n)
=&
\frac{1}{\Gamma(\frac{\lambda+\nu-n+1}{2})\Gamma(\frac{\lambda-\nu}{2})}
\ka\lambda\nu(x,x_n)
\\
=& \frac{1}{\Gamma(\frac{\lambda+\nu-n+1}{2})\Gamma(\frac{\lambda-\nu}{2})}
   |x_n|^{\lambda+\nu-n} (|x|^2 + x_n^2)^{-\nu}.  
\end{align*}
For $(\lambda,\nu) \in \mathbb{C}$
 such that $\nu-\lambda =2l$
 ($l \in {\mathbb{N}}$), 
 we have from \eqref{eqn:Cop}:
\begin{align}
\KC\lambda \nu
=&
\sum_{j=0}^l 
\frac
{2^{2l-2j} \prod_{i=1}^{l-j}(\frac{\lambda+\nu-n-1}{2}+i)}{j!(2l-2j)!}
(\Delta_{{\mathbb{R}}^{n-1}}^j \delta(x_1,\cdots,x_{n-1}))
\delta^{(2l-2j)}(x_n)
\notag
\\
=&
\widetilde C_{2l}^{\lambda-\frac{n-1}2}
(-\Delta_{{\mathbb{R}}^{n-1}}, \frac{\partial}{\partial x_n})
\,
\delta(x_1,\cdots,x_{n-1})
\delta(x_n).  
\label{eqn:KC}
\end{align}
\begin{proposition}
\label{prop:4.3}
{\rm{1)}}\enspace
The tempered distribution 
 ${\mathcal{F}}_{\mathbb{R}} \KA{\lambda}{\nu}
 \in {\mathcal{S}}'(\mathbb{R}^n)$
 is a real analytic function 
 (in particular, 
 locally integrable)
 in the open subset
\[
   \{(\xi,\xi_n)\in \mathbb{R}^{n-1}\oplus \mathbb{R}
  :
  |\xi|>|\xi_n|\}, 
\]
where $|\xi|=(\xi_1^2 + \cdots + \xi_{n-1}^2)^{\frac 1 2}$.  
By the analytic continuation
 (cf. \eqref{eqn:FRc}), 
 we have
 for $|\zeta|>|\zeta_n|$
\begin{equation}
\label{eqn:FcA}
({\mathcal{F}}_c \KA{\lambda}{\nu})(\zeta, \zeta_n)
=\frac{\pi^{\frac{n-1}{2}}e^{\frac{\pi i} 2 (\nu-\lambda)}|\zeta|^{\nu-\lambda}}{\Gamma(\nu)2^{\nu-\lambda}}
{}_2F_1
(\frac{\lambda-\nu}{2}, \frac{\lambda+\nu+1-n}{2};
\frac 1 2;-\frac{\zeta_n^2}{|\zeta|^2}).  
\end{equation}
\par\noindent
{\rm{2)}}\enspace
Suppose $\nu -\lambda = 2l$
 ($l \in {\mathbb{N}}$).  
Then 
\[
  ({\mathcal{F}}_c \KC{\lambda}{\nu})(\zeta,\zeta_n)
=\widetilde {C}_{2l}^{\lambda-\frac{n-1}{2}}
 (-|\zeta|^2, \zeta_n).  
\]
\end{proposition}
\begin{proof}
1) \enspace
We use the integration formula
\begin{equation*}
   {\mathcal{F}}_{{\mathbb{R}}^n}
     \ka \lambda\nu(\xi, \xi_n)
   =
   \frac{2^{-\nu+\frac {n+1} 2}\pi^{\frac {n-1} 2}}
        {\Gamma(\nu)|\xi|^{-\nu+\frac {n-1} 2}}
   \int_{{\mathbb{R}}}
   |t|^{\lambda-\frac {n+1} 2}
   K_{-\nu+\frac {n-1} 2}(|t\xi|)e^{-it \xi_n}d t, 
\end{equation*}
where $K_{\mu}(t)$ denotes the $K$-Bessel function.  
We then apply the following integration formula
(see \cite[6.699.4]{GR}): 
\begin{align*}
&\int_0^\infty x^\gamma K_\mu(ax) \cos(bx) dx
\\
&=
2^{\gamma-1} a^{-\gamma-1}
\Gamma(\frac{\mu+\gamma+1}{2})
\Gamma(\frac{1+\gamma-\mu}{2})
{}_{2}F_1(\frac{\mu+\gamma+1}{2}, \frac{1+\gamma-\mu}{2};
        \frac{1}{2}; - \frac{b^2}{a^2} )
\end{align*}
for $\operatorname{Re}(-\gamma\pm\mu) < 1$,
$\operatorname{Re} a > 0$, $b > 0$.
\par\noindent
2)\enspace
Clear from \eqref{eqn:KC} and the definition
 of ${\mathcal {F}}_c$.  
\end{proof}
\begin{proof}
[Proof of Theorem \ref{thm:AC}]
By using the following formula
of the Gegenbauer polynomial 
 of even degree
\[
C_{2l}^{\mu}(x)
=\frac{(-1)^l \Gamma(l+\mu)}{l ! \Gamma(\mu)}
 {}_{2}F_{1}(-l, l+\mu;\frac 1 2;x^2), 
\]
we get
\[
{\mathcal{F}}_c \KA{\lambda}{\nu}
=
\frac{(-1)^l l! \pi^{\frac{n-1}{2}}}
     {2^{2l}\Gamma(\nu)}
{\mathcal{F}}_c \KC{\lambda}{\nu}
\]
for $|\zeta| > |\zeta_n|$.  

In view that $\operatorname{Supp}\KA{\lambda}{\nu}
\subset \{0\}$
 for $\nu-\lambda \in 2 {\mathbb{N}}$
 (\cite{KS}), 
 both ${\mathcal{F}}_c \KA{\lambda}{\nu}$
 and ${\mathcal{F}}_c \KC{\lambda}{\nu}$
 are holomorphic functions
 on ${\mathbb{C}}^n$.  
Hence Theorem \ref{thm:AC} follows.  
\end{proof}
\begin{remark}
The assumption \eqref{eqn:PNPG}
 of Theorem \ref{thm:Fcont}
 is satisfied,
 and therefore,
 $\operatorname{Hom}_{G'}
(C^{\infty}(G/P,{\mathcal{L}}_{\lambda}),
C^{\infty}(G'/P',{\mathcal{L}}_{\nu}))$
 can be identified
 with the following subspace
 of the Schwartz distributions:
\[
\Sol(\lambda,\nu)^{\wedge}=
\{F \in {\mathcal{S}}'({\mathbb{R}}^n)
:
\text{$F$ satisfies the following three equations}
\}.  
\]
\begin{align}
&F(m \cdot)=F(\cdot)
\quad
\text{for $m \in O(n-1)\times O(1)$},
\notag
\\
&(
\sum_{i=1}^{n}
\zeta_i \frac {\partial}{\partial \zeta_i}
+\lambda-\nu)
)F=0,
\notag
\\
&
(\nu\frac{\partial}{\partial \zeta_j}
-\frac 1 2 \Delta_{\mathbb{R}^n}\zeta_j)
F=0
\quad
(1 \le j \le n-1).  
\label{eqn:fund}
\end{align}
The differential operators in the last equation 
are the {\it{fundamental differential operators}}
 on the light cone
 \cite[Chapter 2]{xkmanoAMS}
 (or Bessel operators in the context of Jordan algebras).  
The heart of the F-method
 is that this differential equation
 explains 
 why the Gauss hypergeometric functions
 (and the Gegenbauer polynomials
 as special cases)
 arise 
 in the formula
 of ${\mathcal{F}}_c \KA\lambda\nu$
 and ${\mathcal{F}}_c \KC\lambda\nu$
 in Proposition \ref{prop:4.3}.  
(In \cite{Juhl}, 
 the relation between $\tcC{\lambda}{\nu}$
 and the Gegenbauer polynomial was pointed out,
 but the proof was based on the comparison
 of coefficients
 determined by recurrence relations.)
\end{remark}
We recall
 that the Riesz potential 
\[
\KT{n-\lambda}{\lambda}(x,x_n)
:=\frac{1}{\Gamma(\lambda+\frac n 2)}
(x_1^2+ \cdots +x_n^2)^{\lambda}
\]
gives the normalized Knapp--Stein intertwining operator
 by
\[
\tcT{n-\lambda}{\lambda}
:C^{\infty}({\mathbb{R}}^n) \to C^{\infty}({\mathbb{R}}^n),
\quad
f\mapsto \KT{n-\lambda}{\lambda} \ast f.  
\]
Then $\tcT{n-\lambda}{\lambda}$ 
 depends holomorphically on $\lambda \in {\mathbb{C}}$
 and satisfies
\[
\varpi_{\lambda}^G(h)\circ \tcT{n-\lambda}{\lambda}
=
\tcT{n-\lambda}{\lambda}\circ\varpi_{n-\lambda}^G(h)
\quad
\text{for all }h \in G.  
\]
Here are the functional equations
 among the three operators:
our operators $\tcA {\lambda}{\nu}$, 
 the Knapp--Stein operators
 $\tcT {\nu}{m-\nu}$
 for $G'$
 and $\tcT {n-\lambda}\lambda$
for $G$.  

\begin{theorem}
[{\cite{KS}}]
\label{thm:TAATE}
Let $m=n-1$.
\begin{alignat}{2}
&
\T{\nu}{m-\nu} \circ \tcA{\lambda}{\nu}
&&= \frac{\pi^{\frac{m}{2}}}{\Gamma(\nu)} \tcA{\lambda}{m-\nu}.
\label{eqn:TAAnewE}
\\
&\tcA{\lambda}{\nu} \circ \T{n-\lambda}{\lambda}
&&= \frac{\pi^{\frac{n}{2}}}{\Gamma(\lambda)} \tcA{n-\lambda}{\nu}.
\label{eqn:ATTnewE}
\end{alignat}
\end{theorem}
\begin{proof}
[Heuristic idea of a proof based on the F-method]
First,
 we compute the Fourier transform
 of the Riesz potential,
 and obtain 
\begin{equation}
{\mathcal{F}}_c(\KT{n-\lambda}{\lambda})
(\zeta,\zeta_n)
= \frac{e^{\frac{\pi i}{2}(2 \lambda-n)}\pi^{\frac n 2}}{2^{2 \lambda-n}\Gamma(\lambda)}(|\zeta|^2+\zeta_n^2)^{\lambda-\frac n 2}.  
\label{eqn:FcK}
\end{equation}
Combining \eqref{eqn:FcA} and \eqref{eqn:FcK}, 
we would have the following identity
 of holomorphic functions
 on $\{(\zeta, \zeta_n)\in {\mathbb{C}}^n:|\zeta|>|\zeta_n|\}$
 as analytic continuation:
\begin{align*}
&{\mathcal{F}}_c(\KA{\lambda}{\nu}\ast \KT{n-\lambda}{\lambda})
\\
=&
{\mathcal{F}}_c(\KA{\lambda}{\nu})
{\mathcal{F}}_c(\KT{n-\lambda}{\lambda})
\\
=&
\frac{\pi^{n-\frac 1 2} 
e^{\frac{\pi i}{2}(\lambda+\nu-n)}
(|\zeta|^2+\zeta_n^2)^{\lambda-\frac n2}
|\zeta|^{\nu-\lambda}}
{\Gamma(\lambda)\Gamma(\nu)2^{\lambda+\nu-n}}
{}_2F_1
(\frac{\lambda-\nu}{2}, \frac{\lambda+\nu+1-n}{2};
\frac 1 2;-\frac{\zeta_n^2}{|\zeta|^2}).  
\end{align*}
Then the desired functional equation \eqref{eqn:TAAnewE} would 
 be reduced to Kummer's relation on Gauss hypergeometric functions:
\[
F(\alpha,\beta;\gamma;z)
=
(1-z)^{\gamma-\alpha-\beta} F(\gamma-\alpha, \gamma-\beta; \gamma; z).  
\]
The identity \eqref{eqn:ATTnewE} is similar and simpler.  
\end{proof}
An advantage of the F-method indicated as above
 is that we can discover the functional identities
 such as \eqref{eqn:TAAnewE} and \eqref{eqn:ATTnewE}
as a disguise
 of a very simple and classical identity
 (i.e. Kummer's relation), 
 and the proof does not depend heavily on representation theory.
On the other hand, since the convolution (or the multiplication)
 of two Schwartz distributions are not well-defined in the usual sense
 in general,
 a rigorous proof
 in this direction requires some careful analysis when we deal with 
 such functional equations for non-local operators.   
(For local operators, we do not face with these analytic difficulties.
In this case, we have used the F-method in \cite{KOSS}
 to prove functional identities for differential operators,
 e.g. factorization identities in \cite{Juhl}.)

In \cite{KS}, we take a completely different approach
 based on the uniqueness of symmetry breaking operators
 for generic parameters
 (cf.~\cite{GP,xsunzhu})
 and the evaluation of spherical vectors
 applied by symmetry breaking operators
 for the proof of Theorem \ref{thm:TAATE} and its variants.

\subsection*{Acknowledgments}
%
The author is grateful to J.-L.~Clerc, 
G.~Mano, T.~Matsuki,
 B.~\O rsted, T. Oshima,
 M.~Pevzner, B.~Speh, P.~Somberg, 
V.~Souceck for their collaboration on the papers which are mentioned
 in this article.
Parts of the results were delivered
at the conference,
\textsl{the Interaction of Geometry and Representation Theory: Exploring New Frontiers}
in honor of Michael Eastwood's 60th birthday,
organized by 
Andreas Cap,  Alan Carey,  A. Rod Gover,  C. Robin Graham, and Jan Slovak,
at ESI, Vienna, 10--14 September 2012. 
Thanks are also due to referees for reading carefully the manuscript.
This work is partially
 supported by
the Institut des Hautes \'Etudes Scientifiques (Bures-sur-Yvette)
and Grant-in-Aid for Scientific Research 
 (A) (25247006) JSPS.


\begin{thebibliography}{99}
\bibitem{xatiyah}
M.~Atiyah, 
\emph{Resolution of singularities and division of
distributions}, 
Comm. Pure Appl. Math.,
{\bf{23}},
(1970)
145--150.

\bibitem{xbege}
I.~N.~Bernstein, S.~I.~Gelfand, 
\emph{Meromorphic property of the functions
$P^{\lambda}$},
Funktsional Anal, i Prilozhen.,
{\bf{3}},  
(1969)
84--85.  

\bibitem{xbz}
B.~Binegar, R.~Zierau,
\emph{Unitarization of a singular representation
 of $SO(p,q)$},
Comm. Math. Phys.,
{\bf{138}},
(1991)
245--258.  

\bibitem{xclare}
P.~Clare,
\emph{On the degenerate principal series of complex symplectic groups}, 
J. Funct. Anal. 
{\bf{262}},
(2012)
4160--4180. 

\bibitem{CKOP}
J.-L.~Clerc, 
T.~Kobayashi,
B.~{\O}rsted, and M.~Pevzner, 
\emph{Generalized
Bernstein--Reznikov integrals},  
Math. Ann.,
{\bf{349}},
(2011) 
\href{http://dx.doi.org/10.1007/s00208-010-0516-4}
{395--431}.

\bibitem{C75}
H.~Cohen, 
\emph{Sums involving the values at
negative integers of $L$-functions of quadratic characters},
Math. Ann.,
\textbf{217},
(1975)
271--285.

\bibitem{CMZ}
P.~B. Cohen, Y.~Manin, D.~Zagier, 
\emph{Automorphic pseudodifferential operators}, 
Progr. Nonlinear Differential Equations Appl., 
\textbf{26}, Birkh\"auser, 1997
pp. 17--47.

\bibitem{EGR}
M.~Eastwood, R.~Graham, 
\emph{Invariants of conformal densities},
Duke Math.~J., \textbf{63} (1991), 633--671. 

\bibitem{GGV}
I. M. Gelfand, M. I. Graev, N. Ya. Vilenkin,
Generalized functions. Vol. 5: Integral geometry and representation theory,
Academic Press, New York, 
1966, xvii+449 pp. 

\bibitem{GR}
I.~S.~Gradshteyn, I.~M.~Ryzhik,
Table of Integrals, Series, and Products,
Seventh edition. Elsevier,
 2007.
xlviii+1171 pp.

\bibitem{GP} 
B.~Gross, D.~Prasad,  
On the decomposition of a representations of SO$_n$
 when restricted to SO$_{n-1}$, 
 Canad. J. Math. {\bf{44}}
 (1992), 
 974--1002.  

\bibitem{xhaja}
M.~Harris, H.~P.~Jakobsen,
\emph{Singular holomorphic representations
 and singular modular forms}, 
Math. Ann.,
{\bf{259}},
(1982)
227--244.  

\bibitem{xhuzh}
J.-S. Huang, C.-B. Zhu,
\emph{On certain small representations of indefinite
orthogonal groups}, 
Represent. Theory,
{\bf{1}}, 
(1997)
190--206.

\bibitem{Juhl}
A. Juhl, 
\emph{ Families of conformally covariant differential operators, $Q$-curvature and holography.} 
Progr. Math., 
\href{http://link.springer.com/book/10.1007/978-3-7643-9900-9/page/1}
{\bf{275}}. Birkh\"auser, 
2009. 

\bibitem{xkaos}
M.~Kashiwara, T.~Oshima, 
\emph{Systems of differential equations
 with regular singularities
 and their boundary value problems},
Ann. of Math., 
{\bf{106}},
(1977)
145--200.  

\bibitem{xkiai}
T.~Kobayashi, 
{\emph{Discrete decomposability of the restriction of
             $A_{\frak q}(\lambda)$
            with respect to reductive subgroups and its applications}}, 
Invent. Math.,
{\bf{117}}, 
(1994)
\href{http://dx.doi.org/10.1007/BF01232239}{181--205};
Part II, 
Ann. of Math., 
{\bf {147}}, 
(1998)
\href{http://dx.doi.org/10.2307/120963}{709--729};
Part III, Invent. Math., {\bf{131}}, (1998)
\href{http://dx.doi.org/10.1007/s002220050203}{229--256}.  

\bibitem{Ksuron}
T.~Kobayashi,
\emph{{I}ntroduction to harmonic analysis
 on real spherical homogeneous spaces},
Proceedings of the 3rd Summer School on Number Theory
\lq\lq{Homogeneous Spaces and Automorphic Forms}\rq\rq\
in Nagano (F.~Sato, ed.), 1995, 22--41 (in Japanese).  

\bibitem{xkcheck}
T.~Kobayashi,
\emph{{Conformal geometry and global solutions to the {Y}amabe 
equations on classical
  pseudo-{R}iemannian manifolds}}, 
   Proceedings of the 22nd Winter School
  ``Geometry and Physics'' (Srn\'\i, 2002), 
Rend. Circ. Mat. Palermo (2)  Suppl. \textbf{71}, 
 2003, pp.~
\href{http://www.kurims.kyoto-u.ac.jp/preprint/file/RIMS1365.pdf}
{15--40}.

\bibitem{K08}
T.~Kobayashi, 
\emph{Multiplicity-free theorems of the restrictions
 of unitary highest weight modules
 with respect to reductive symmetric pairs}, 
pp. 45--109, Progr. Math., 
\href{http://dx.doi.org/10.1007/978-0-8176-4646-2_3}{\textbf{255}}, 
Birkh\"auser, Boston, 2008.

\bibitem{xkzuckerman}
T.~Kobayashi, 
\emph{Branching problems of Zuckerman derived functor modules},
Contemp. Math., 
{\bf{557}}, 
pp. 23--40. Amer. Math. Soc., 2011.
(available at 
\href{http://arxiv.org/abs/1104.4399}{arXiv:1104.4399})

\bibitem{K12} 
T.~Kobayashi, 
\emph{Restrictions of generalized Verma modules to symmetric 
pairs}, 
Transform. Group, 
{\bf{17}},
 (2012)
\href{http://dx.doi.org/10.1007/s00031-012-9180-y}{523--546}.


\bibitem{xkhelgason}
T.~Kobayashi, 
\emph{$F$-method for constructing equivariant 
 differential operators},
 Contemp. Math., {\bf {598}}, 
\href{http://dx.doi.org/10.1090.conm/598/11998}{pp.~141--148},
 Amer. Math. Soc., 2013
(available at 
\href{http://arxiv.org/abs/1212.6862}{arXiv:1212.6862}
).  

\bibitem{mfbundl}
T.~Kobayashi, 
Propagation of multiplicity-freeness property for
holomorphic vector bundles, Progr. Math. {\bf{306}}, 
Birkh\"{a}user, 
2013, pp. \href{http://dx.doi.org/10.1007/978-1-4614-7193-6_6}{113--140},
(available at 
\href{http://uk.arxiv.org/abs/math.DG/0607004}{arXiv:0607004}
).


\bibitem{xkmanoAMS}
T.~Kobayashi, G.~Mano,
The Schr{\"o}dinger model for the minimal
representation of the indefinite orthogonal group $O(p,q)$,
Mem. Amer. Math. Soc. (2011), \textbf{212},
\href{http://dx.doi.org/10.1090/S0065-9266-2011-00592-7}{no. 1000}
, vi+132 pp.

\bibitem{xKMt}
T.~Kobayashi, T.~Matsuki,
Classification of multiplicity finite symmetric pairs,
in preparation.

\bibitem{xkors}
T.~Kobayashi, B.~{\O}rsted,
\emph{{Analysis on the
  minimal representation
  of\/ {${\rm O}(p,q)$}.
 {\rm{Part I,}} 
}}
Adv. Math.,
\textbf{180}, (2003) 
\href{http://dx.doi.org/10.1016/S0001-8708(03)00012-4}{486--512};
{\rm{Part II}},
{\it{ibid}}, 
\href{http://dx.doi.org/10.1016/S0001-8708(03)00013-6}{513--550};
{{\rm{Part III}}}, 
{\it{ibid}},
\href{http://dx.doi.org/10.1016/S0001-8708(03)00014-8}{551--595}.

\bibitem{KOP}
T.~Kobayashi, B.~{\O}rsted, M.~Pevzner,
\emph{Geometric analysis on small unitary representations
 of $GL(n,\mathbb R)$}, 
J. Funct. Anal., 
{\bf{260}}, 
(2011)
\href{http://dx.doi.org/10.1016/j.jfa.2010.12.008} 
{1682--1720}.

\bibitem{KOSS}
T.~Kobayashi, B.~{\O}rsted,
P.~Somberg, V.~Sou\v{c}ek, 
\emph{Branching laws for Verma modules and applications in
parabolic geometry}, Part I, preprint, 37 pages,
\href{http://arxiv.org/abs/1305.6040}{arXiv:1305.6040}.

\bibitem{xktoshima}
T.~Kobayashi, T.~Oshima, 
\emph{Finite multiplicity theorems for induction and restriction}, 
Advances in Mathematics, \textbf{248}, (2013), 921--944, \href{http://dx.doi.org/10.1016/j.jfa.2010.12.008}{doi:10.1016/j.aim.2013.07.015}, 
(available at
\href{http://arxiv.org/abs/1108.3477}{arXiv:1108.3477}).

\bibitem{KP}
T.~Kobayashi, M.~Pevzner,
\emph{Rankin--Cohen operators for symmetric pairs},
preprint, 53pp.
\href{http://arxiv.org/abs/1301.2111}{arXiv:1301.2111}.  

\bibitem{KS}
T.~Kobayashi, B.~Speh,
\emph{Symmetry breaking for representations  of rank
one orthogonal groups}, preprint.

\bibitem{xkostant}
B.~Kostant,
The vanishing of scalar curvature
 and the minimal representation of $SO(4,4)$, 
Progr. Math., 
{\bf{92}}, 
Birkh{\"a}user, 
1990, pp. 85--124.  

\bibitem{xkramer}
M.~Kr{\"a}mer, 
Multiplicity free subgroups
 of compact connected Lie groups,
 Arch. Math. (Basel)
 {\bf{27}} (1976), 28--36.  

\bibitem{Ma}
H. Matumoto,
On the homomorphisms between scalar generalized Verma modules,
preprint 46 pp. 
\href{http://arxiv.org/abs/1205.6748}
{arXiv:1205.6748}.  

\bibitem{xrankin}
R.~A.~Rankin,
\emph{The construction of automorphic forms from the derivatives
 of a given form},
J. Indian Math. Soc.,  
{\bf{20}},
(1956)
103--116.  

\bibitem{xsunzhu}
B.~Sun, C.-B.~Zhu,
\emph{Multiplicity one theorems:
the Archimedean case},
Ann. of Math., 
{\bf{175}}, 
(2012)
23--44.  

\end{thebibliography}
\end{document}